\documentclass[11pt,oneside]{amsart}
\usepackage[top=25mm, bottom=25mm, left=20mm, right=20mm]{geometry}
\usepackage{amsfonts,amssymb,amsmath,amsthm,amsxtra}
\usepackage[T1]{fontenc}
\usepackage[utf8]{inputenc}
\usepackage{bm}
\usepackage{mathtools}
\usepackage{dsfont}
\usepackage{mathrsfs}
\usepackage{enumitem}
\allowdisplaybreaks

\usepackage{graphics}
\usepackage{hyperref}

\numberwithin{equation}{section}
\newtheorem{theorem}[equation]{Theorem}

\newtheorem{lemma}[equation]{Lemma}
\newtheorem{proposition}[equation]{Proposition}
\theoremstyle{definition}

\newtheorem{definition}[equation]{Definition}

\author{Leonidas Daskalakis} 
\address[Leonidas Daskalakis]{Institute of Mathematics,
Polish Academy of Sciences,
\'Sniadeckich 8,
00-656 Warszawa, Poland}
\email{ldaskalakis@impan.pl}
\begin{document}
\title{Pointwise ergodic theorems for non-conventional bilinear averages along $(\lfloor n^c\rfloor,-\lfloor n^c\rfloor)$}
\maketitle
\begin{abstract}For every $c\in(1,23/22)$ and every probability dynamical system $(X,\mathcal{B},\mu,T)$ we prove that for any $f,g\in L^{\infty}_{\mu}(X)$ the bilinear ergodic averages
\[
\frac{1}{N}\sum_{n=1}^Nf(T^{\lfloor n^c\rfloor}x)g(T^{-\lfloor n^c\rfloor}x)\quad\text{converge for $\mu$-a.e. $x\in X$.}
\]
In fact, we consider more general sparse orbits $(\lfloor h(n)\rfloor,-\lfloor h(n)\rfloor)_{n\in\mathbb{N}}$, where $h$ belongs to the class of the so-called $c$-regularly varying functions. This is the first pointwise result for bilinear ergodic averages taken along deterministic sparse orbits where modulation invariance is present.
\end{abstract}
\section{Introduction}
In 1990 Bourgain \cite{BDR} established that for any measure-preserving system $(X,\mathcal{B},\mu,T)$ with finite measure and any $f,g\in L_{\mu}^{\infty}(X)$ we have that
\begin{equation}\label{BourgAv}
\lim_{N\to\infty}\frac{1}{N}\sum_{n=1}^Nf(T^nx)g(T^{-n}x)\quad\text{converges for $\mu$-a.e. $x\in X$.}
\end{equation}
By standard arguments relying on H\"older's inequality and the maximal ergodic theorem this result can be immediately extended to functions $f\in L^{p_1}_{\mu}(X)$ and $g\in L^{p_2}_{\mu}(X)$ where $\frac{1}{p_1}+\frac{1}{p_2}\le 1$, and by appealing to the bilinear inequality of Lacey \cite{Lacey32} one obtains the result even for exponents satisfying $\frac{1}{p_1}+\frac{1}{p_2}<3/2$, see also Theorem~1.16 in \cite{KTM}.

Recently, there has been substantial progress in understanding pointwise phenomena in ergodic theory in the bilinear as well as the multilinear setting. More specifically, Krause, Tao and Mirek \cite{KTM} established pointwise convergence for the following ergodic averages
\[
\frac{1}{N}\sum_{n=1}^Nf(T^nx)g(T^{P(n)}x)\text{,}
\]
where $P\in \mathbb{Z}[n]$ has degree $d\ge 2$. In 2024 the first multilinear result in that direction appeared \cite{KPMW}, and pointwise convergence was established for averages of the form
\[
\frac{1}{N}\sum_{n=1}^Nf_1(T_1^{P_1(n)}x)\cdots f_k(T_k^{P_k(n)}x)\text{,}
\]
where $P_1,\dotsc,P_k\in\mathbb{Z}[n]$ have distinct degrees and $T_1,\dotsc,T_k$ are commuting invertible measure-preserving transformations.

Both results rely heavily on the fact that, loosely speaking, the orbits considered introduce a certain kind of ``curvature''; this is visible from the distinct degrees assumption. 

Currently, establishing pointwise convergence for averages of the form
\begin{equation}\label{toyprob}
\frac{1}{N}\sum_{n=1}^Nf(T^{P(n)}x)g(T^{-P(n)}x)\text{,}
\end{equation}
where $P\in\mathbb{Z}[n]$ has degree $d\ge 2$, seems to be out of reach, and the techniques from the aforementioned works cannot be applied since ``curvature'' is not present here. In some sense, such ergodic averages share some similarities with Bourgain's bilinear averages $\ref{BourgAv}$, and modulation invariance phenomena ought to be handled for addressing such pointwise convergence problems.

We establish the following bilinear pointwise convergence result in the direction of $\ref{toyprob}$, for $P(n)=\lfloor n^c\rfloor$ and $c\in(1,23/22)$.
\begin{theorem}\label{MainNc}
Assume $c\in(1,23/22)$. Let $(X,\mathcal{B},\mu)$ be a probability space and $T\colon X\to X$ an invertible $\mu$-invariant transformation. Then for every $f,g\in L_{\mu}^{\infty}(X)$ we have that
\begin{equation}\label{ptconv}
\lim_{N\to\infty}\frac{1}{N}\sum_{n=1}^Nf(T^{\lfloor n^c\rfloor} x)g(T^{-\lfloor n^c\rfloor} x)\text{ exists for $\mu$-a.e. $x\in X$.}
\end{equation}
\end{theorem}
To the best of the author's knowledge this is the first pointwise bilinear result with deterministic sparse orbits and with modulation invariance present. Notably, a randomized variant of fractional powers in the spirit of the theorem above has been considered in \cite{probvariant} and although the technical part of our argument is quite more involved we encourage the reader to compare the strategy described in subsection~2.1 of the aforementioned paper with our own described in subsection~$\ref{strat}$.

We derive this result as a corollary of more general theorem allowing us to replace $\lfloor n^c\rfloor$ in $\ref{ptconv}$ with any of the following orbits 
\begin{equation}\label{examples1}
\lfloor n^c\log^{a}n\rfloor,\quad \lfloor n^ce^{a\log^b n}\rfloor,\quad \lfloor n^c\underbrace{\log\circ \cdots \circ \log n}_{k \text{ times}}\rfloor\text{,}
\end{equation}
where $a\in\mathbb{R}$, $b\in(0,1)$, $k\in\mathbb{N}$ are fixed. In fact, for $a>0$ even the case $c=1$ will be addressed for the sequences above. Before stating this theorem, we introduce the so-called $c$-regularly varying functions.
\begin{definition}[$c$-regularly varying functions] Let $c\in[1,2)$ and $x_0\ge 1$. We define the class of $c$-regularly varying functions $\mathcal{R}_c$ as the set of all functions $h\in\mathcal{C}^3\big([x_0,\infty)\to[1,\infty)\big)$ such that the following conditions hold:
\begin{itemize}
\item[i)]$h'>0$, $h''> 0$\text{ and }
\item[ii)] the function $h$ is of the form 
\[
h(x)=Cx^c\exp\bigg(\int_{x_0}^x\frac{\vartheta(t)}{t}dt\bigg)\text{,}
\]
where $C$ is a positive constant and $\vartheta\in\mathcal{C}^3\big([x_0,\infty)\to\mathbb{R}\big)$ satisfies
\[
\vartheta(x)\to 0\,\,\text{, }x\vartheta'(x)\to 0\,\,\text{, }x^2\vartheta''(x)\to 0\,\,\text{, }x^3\vartheta'''(x)\to 0\,\,\text{ as }x\to \infty\text{.}
\]
\item[iii)] If $c=1$, then $\vartheta$ will additionally be assumed to be positive, decreasing, and for every $\varepsilon>0$
\[
\vartheta(x)^{-1}\lesssim_{\varepsilon}x^{\varepsilon}\quad\text{and}\quad \lim_{x\to\infty}xh(x)^{-1}=0\text{.}
\]
Moreover, for $c=1$ we assume
\[
\lim_{x\to\infty}\frac{x\vartheta'(x)}{\vartheta(x)}=0\text{,}\quad\lim_{x\to\infty}\frac{x^2\vartheta''(x)}{\vartheta(x)}=0\text{,}\quad\lim_{x\to\infty}\frac{x^3\vartheta'''(x)}{\vartheta(x)}=0\text{.}
\]
\end{itemize}
\end{definition}
The family of $c$-regularly varying functions has been introduced in \cite{MMR} and \cite{WT11} and one may think of functions in $\mathcal{R}_c$ as appropriate perturbations of the fractional monomial $x^c$.  Notably, there exist $c$-regularly varying functions which do not belong to any Hardy field.
\begin{theorem}\label{MTHM}
Assume $c\in[1,23/22)$ and $h\in\mathcal{R}_c$. Let $(X,\mathcal{B},\mu)$ be a probability space and $T\colon X\to X$ an invertible $\mu$-invariant transformation. Then for every $f,g\in L_{\mu}^{\infty}(X)$ we have that
\begin{equation}\label{ptconvforh}
\lim_{N\to\infty}\frac{1}{N}\sum_{n=1}^Nf(T^{\lfloor h(n)\rfloor} x)g(T^{-\lfloor h(n)\rfloor} x)\text{ exists for $\mu$-a.e. $x\in X$.}
\end{equation}
Moreover, we have that
\begin{equation}\label{thelimit}
\lim_{N\to\infty}\frac{1}{N}\sum_{n=1}^Nf(T^{\lfloor h(n)\rfloor} x)g(T^{-\lfloor h(n)\rfloor} x)=\lim_{N\to\infty}\frac{1}{N}\sum_{n=1}^Nf(T^{n} x)g(T^{-n} x)\quad\text{for $\mu$-a.e. $x\in X$.}
\end{equation}
\end{theorem}
The theorem above immediately implies Theorem~$\ref{MainNc}$ and establishes the analogous results corresponding to any of the orbits mentioned in $\ref{examples1}$ by appropriately choosing $h\in\mathcal{R}_c$. As mentioned earlier, standard arguments relying on H\"older's inequality and on the fact that the maximal function corresponding to the single averages is bounded on $L^p_{\mu}(X)$ for every $p\in(1,\infty]$, see \cite{ncfull} and \cite{WT11}, allow one to deduce $\ref{ptconvforh}$ for $f\in L_{\mu}^{p_1}(X)$, $g\in L_{\mu}^{p_2}(X)$ with $\frac{1}{p_1}+\frac{1}{p_2}< 1$ as a simple corollary of Theorem~$\ref{MTHM}$. Moreover, if $c\in(1,30/29)$, then the maximal operator associated with the single averages along such orbits is of weak-type (1,1), see \cite{WT11}, and with a similar argument one may cover the case $\frac{1}{p_1}+\frac{1}{p_2} \le1$ for such $c$'s.

\subsection{Strategy}\label{strat}  We wish to give a brief description of the strategy of our proof here. In Section~$\ref{mainreductioncrucial}$, we split our averaging operator to two pieces: a main term, which converges to the same limit as Bourgain's averages, see $\ref{BourgAv}$, and an error term. More specifically, we consider our ergodic averages in the following form
\[
B_N(f,g)(x)\coloneqq\frac{1}{|\mathbb{N}_h\cap [1,N]|}\sum_{n\in [1,N]}1_{\mathbb{N}_h}(n)f(T^nx)g(T^{-n}x)\text{,}
\] 
where $\mathbb{N}_h\coloneqq\{\lfloor h(n)\rfloor:\,n\in\mathbb{N}\}$. After exploiting the following formula $1_{\mathbb{N}_h}(n)=\lfloor -\varphi(n)\rfloor-\lfloor -\varphi(n+1)\rfloor$, where $\varphi$ is the compositional inverse of $h$, we may write
\[
1_{\mathbb{N}_h}(n)=\big(\varphi(n+1)-\varphi(n)\big)+\big(\Phi(-\varphi(n+1))-\Phi(-\varphi(n))\big)\text{,}
\]
where $\Phi(x)=\{x\}-1/2$. This decomposes our averages to $B_N(f,g)(x)=M_N(f,g)(x)+E_N(f,g)(x)$, where 
\[
M_N(f,g)(x)\coloneqq \frac{1}{|\mathbb{N}_h\cap [1,N]|}\sum_{n\in [1,N]}\big(\varphi(n+1)-\varphi(n)\big)f(T^nx)g(T^{-n}x)\text{,}
\]
and
\[
E_N(f,g)(x)\coloneqq\frac{1}{|\mathbb{N}_h\cap [1,N]|}\sum_{n\in [1,N]}\big(\Phi(-\varphi(n+1))-\Phi(-\varphi(n))\big)f(T^nx)g(T^{-n}x)\text{.}
\]
The fact that the weights in the averaging operator $M_N$ are appropriately well-behaving allows us to prove that such averages must converge to the same limit as Bourgain's.

It remains to address the error term, namely, to establish that $\lim_{N\to\infty}E_N(f,g)(x)=0$ for $\mu$-a.e. $x\in X$. We begin by exploiting a famous truncated Fourier series expansion for $\Phi$ with uniform bounds on its tail, see Lemma~$\ref{TrFourier}$. This approximation of $\Phi$ induces a decomposition for $E_N$ and we focus on each piece separately. The operator induced by the tail of the Fourier expansion can be treated straightforwardly, while the other requires a significant amount of work. Specifically, certain cancellation in the kernel defining the operator at hand ought to be exploited. This will be done in Section~$\ref{difSection}$ by certain Gowers norm-type bounds. After considering our averages on the integer shift system, we establish that our operator is controlled essentially by the $U^3$-norm of the kernel, and we prove that this norm has appropriate decay. This is done using Fourier-analytic methods since we may view the $U^3$-norm of our kernel as averages of $U^2$-norms of appropriate expressions, and we can immediately use the inverse theorem for the $U^2$-norm. The task of estimating the $U^3$-norm of our kernel reduces to the study of certain exponential sums. Finally, in Section~$\ref{finSection}$, we discuss how such estimates over the integers allow us to conclude.
\subsection{Notation} For any $x\in\mathbb{R}$ we use the standard notation
\[
\lfloor x\rfloor=\max\{n\in\mathbb{Z}:\,n\le x\}\text{,}\quad\{x\}=x-\lfloor x\rfloor\text{,} \quad\|x\|=\min\{|x-n|:\,n\in\mathbb{Z}\}\text{.}
\]
If $A,B$ are two nonnegative quantities, we write $A\lesssim B$ or $B \gtrsim A$  to denote that there exists a positive constant $C$ such that $A\le C B$. Whenever $A\lesssim B$ and $A\gtrsim B$, we write $A\simeq B$. Throughout the paper all the implicit constants appearing may depend on a fixed choice of $h\in\mathcal{R}_c$. We denote $e^{2\pi i x}$ by $e(x)$. For any $N\in\mathbb{N}$ we let $\mathbb{N}_{\ge N}\coloneqq\{n\in\mathbb{N}:\,n\ge N\}$ and $[N]\coloneqq\{1,2,\dotsc,N\}$. Let us note that $h(x)$ is not defined for $x<x_0$ but we abuse notation; we can let $h(x)$ take arbitrary values for $x\in[1,x_0]$ and all our results and estimates remain the same. Given a measurable space $(X,\mathcal{B})$, we call a function $f\colon X\to \mathbb{C}$ $1$-bounded if $f$ is measurable and $|f|\le 1$. For every function $f\colon\mathbb{Z}\to\mathbb{C}$ and $h_1\in\mathbb{Z}$ we define the difference function $\Delta_{h_1}f(x)=f(x)\overline{f(x+h_1)}$, and for every $s\in\mathbb{N}$ and $h_1,\dotsc,h_s\in\mathbb{Z}$ we define $\Delta_{h_1,\dotsc,h_s}f(x)=\Delta_{h_1}\dots\Delta_{h_s}f(x)$. For every $s\in\mathbb{N}_{\ge2}$ and every finitely supported $f\colon\mathbb{Z}\to\mathbb{C}$ we define the Gowers $U^s$-norm by
\[
\|f\|_{U^s}=\bigg(\sum_{x,h_1,\dotsc,h_s\in\mathbb{Z}}\Delta_{h_1,\dotsc,h_s}f(x)\bigg)^{1/2^s}\text{.}
\]
Finally, for every $N\in\mathbb{N}$ we define 
\[
\mu_N(n)=\frac{|\{(h_1,h_2)\in[N]:\,h_1-h_2=n\}|}{N^2}\text{,}
\]
and we note that $\mu_N(n)\lesssim N^{-1}1_{[-N,N]}(n)$.
\section*{Acknowledgments}
The author would like to thank Ben Krause, Michael Lacey and Christoph Thiele for several insightful conversations as well as Nikos Frantzikinakis and Borys Kuca for helpful comments and feedback. I would also like to thank Mariusz Mirek for useful discussions and for his constant support and encouragement.
\section{Main reduction}\label{mainreductioncrucial}
We fix $c\in[1,23/22)$ and $h\in\mathcal{R}_c$, and all implied constants may depend on these choices. We denote by $\varphi$ the compositional inverse of $h$, and we remind the reader that $\mathbb{N}_h=\{\lfloor h(n)\rfloor:\,n\in\mathbb{N}\}$. We note that we use the basic properties of $h$ and $\varphi$ as described in Lemma~2.6  and Lemma~2.14 from \cite{MMR}. Before performing our main reduction, we collect some standard useful results for handling fractional powers and for the sake of clarity we state Toeplitz theorem, a standard result which allows us to change weights in our averaging operators. 
\begin{lemma}\label{allowedtopass}
Assume $c\in[1,2)$, $h\in\mathcal{R}_c$ and $\varphi$ is its compositional inverse. Then there exists a positive constant $C=C(h)$ such that
\begin{equation}\label{indchange}
1_{\mathbb{N}_h}(n)=\lfloor -\varphi(n)\rfloor-\lfloor -\varphi(n+1)\rfloor\quad\text{for all $n\ge C$.}
\end{equation}
\end{lemma}
\begin{proof}
The proof is standard, see for example Lemma~2.12 \cite{MMR}.
\end{proof}

\begin{lemma}\label{TrFourier}
For every $M\in\mathbb{N}_{\ge 2}$ we have that 
\begin{equation}\label{splittored}
\Phi(x)\coloneqq\{x\}-1/2=\sum_{0<|m|\le M}\frac{1}{2\pi im}e(-mx)+g_M(x)\text{,}
 \end{equation}
 with $g_M(x)=O\Big(\min\Big\{1,\frac{1}{M\|x\|}\Big\}\Big)$. We also have that 
 \begin{equation}\label{errorgm}
 \min\bigg\{1,\frac{1}{M\|x\|}\bigg\}=\sum_{m\in\mathbb{Z}}b_me(mx)\text{,}
\end{equation}
 where $|b_m|\lesssim \min\Big\{\frac{\log M}{M},\frac{1}{|m|},\frac{M}{|m|^2}\Big\}$, and all the implied constants are absolute. 
\end{lemma}
\begin{proof}
See Section~2 in \cite{PHI}, or page 260 in \cite{WT11}.
\end{proof}
\begin{theorem}[Toeplitz theorem]\label{SBParts}
For every $N\in\mathbb{N}$ let $(c_{N,k})_{k\in[N]}$ be real numbers. Assume that the following conditions hold:
\begin{itemize}
\item[\normalfont{(i)}] For every $k\in\mathbb{N}$ we have that $\lim_{N\to\infty}c_{N,k}=0$,
\item[\normalfont{(ii)}] $\lim_{N\to\infty}\sum_{k=1}^Nc_{N,k}=1$,
\item[\normalfont{(iii)}] $\sup_{N\in\mathbb{N}}\sum_{k=1}^N|c_{N,k}|<\infty$.
\end{itemize}
Then for any sequence $(a_n)_{n\in\mathbb{N}}$ such that $\lim_{n\to\infty}a_n=a\in\mathbb{C}$, we have that $\lim_{N\to\infty}\sum_{k=1}^Nc_{N,k}a_k=a$.
\end{theorem}
\begin{proof}This result is standard; see for example pages 42-48 in \cite{Toeplitz}.
\end{proof}
We define
\begin{multline}
B_N(f,g)(x)\coloneqq\frac{1}{|\mathbb{N}_h\cap [N]|}\sum_{n\in [N]}1_{\mathbb{N}_h}(n)f(T^nx)g(T^{-n}x)\text{,}
\\
A_N(f,g)\coloneqq\frac{1}{N}\sum_{n\in[N]}f(T^nx)g(T^{-n}x)\text{,}
\\
E_N(f,g)(x)\coloneqq\frac{1}{|\mathbb{N}_h\cap [N]|}\sum_{n\in [N]}\big(\Phi(-\varphi(n+1))-\Phi(-\varphi(n))\big)f(T^nx)g(T^{-n}x)\text{.}
\end{multline}
Bourgain's result together with the first lemma and Toeplitz theorem allows us to prove the following.
\begin{proposition}\label{basicstep1red}
Assume $c\in[1,2)$ and $h\in\mathcal{R}_c$. Let $(X,\mathcal{B},\mu)$ be a probability space and $T\colon X\to X$ an invertible $\mu$-invariant transformation. Then for every $f,g\in L_{\mu}^{\infty}(X)$ we have that
\begin{equation}\label{ptconvinter}
\limsup_{N\to\infty}\big|B_N(f,g)(x)-A_N(f,g)(x)\big|\le\limsup_{N\to\infty}|E_N(f,g)(x)|\quad\text{for $\mu$-a.e. $x\in X$.}
\end{equation}
\end{proposition}
\begin{proof}
We may assume without loss of generality that $f,g$ are $1$-bounded. By Lemma~$\ref{allowedtopass}$ we get
\begin{multline}
B_N(f,g)(x)= \frac{1}{|\mathbb{N}_h\cap [N]|}\sum_{n\in [N]}1_{\mathbb{N}_h}(n)f(T^nx)g(T^{-n}x)
\\
=\frac{1}{|\mathbb{N}_h\cap [N]|}\sum_{n\in [N]}\big(\varphi(n+1)-\varphi(n)\big)f(T^nx)g(T^{-n}x)
\\
+\frac{1}{|\mathbb{N}_h\cap [N]|}\sum_{n\in [N]}\big(\Phi(-\varphi(n+1))-\Phi(-\varphi(n))\big)f(T^nx)g(T^{-n}x)+O\big(\varphi(N)^{-1}\big)\text{,}
\end{multline}
where $\Phi$ is defined in $\ref{splittored}$. Let 
\[
M_N(f,g)(x)\coloneqq \frac{1}{|\mathbb{N}_h\cap [N]|}\sum_{n\in [N]}\big(\varphi(n+1)-\varphi(n)\big)f(T^nx)g(T^{-n}x)\text{,}
\]
and note that to obtain $\ref{ptconvinter}$ it suffices to prove that
\begin{equation}\label{1Goal}\lim_{N\to\infty}M_N(f,g)(x)=\lim_{N\to\infty}A_N(f,g)(x)\quad\text{for $\mu$-a.e. $x\in X$.}
\end{equation}
For $\mu$-a.e. $x\in X$ we know that $\lim_{N\to\infty}A_N(f,g)(x)=L_x$ exists by \cite{BDR}, so let us fix such an $x$. Summation by parts yields
\begin{multline}
|\mathbb{N}_h\cap[N]|M_N(f,g)(x)=\Big(\sum_{n=1}^Nf(T^nx)g(T^{-n}x)\Big)\big(\varphi(N+1)-\varphi(N)\big)
\\
-\sum_{n=1}^{N-1}\Big(\sum_{k=1}^nf(T^kx)g(T^{-k}x)\Big)\Big(\big(\varphi(n+2)-\varphi(n+1)\big)-\big(\varphi(n+1)-\varphi(n)\big)\Big)\text{,}
\end{multline}
and thus
\begin{multline}\label{maintermsbp}
M_N(f,g)(x)=\frac{N\big(\varphi(N+1)-\varphi(N)\big)}{|\mathbb{N}_h\cap[N]|}\Big(\frac{1}{N}\sum_{n=1}^Nf(T^nx)g(T^{-n}x)\Big)
\\
-\frac{1}{|\mathbb{N}_h\cap[N]|}\sum_{n=1}^{N-1}\Big(\frac{1}{n}\sum_{k=1}^nf(T^kx)g(T^{-k}x)\Big)n\Big(\big(\varphi(n+2)-\varphi(n+1)\big)-\big(\varphi(n+1)-\varphi(n)\big)\Big)\text{.}
\end{multline}
For the first summand note that
\begin{equation}\label{phiresults}
\lim_{N\to\infty}\frac{N\big(\varphi(N+1)-\varphi(N)\big)}{|\mathbb{N}_h\cap[N]|}=\lim_{N\to\infty}\frac{N\big(\varphi(N+1)-\varphi(N)\big)}{\varphi(N)}=1/c\text{,}
\end{equation}
To see this, let $\gamma=1/c$, and note that by Lemma~2.6 in \cite{MMR}, there exists a function $\theta$ such that $t\varphi'(t)=\varphi(t)(\gamma+\theta(t))$ and $\lim_{t\to\infty}\theta(t)=0$. By the Mean Value Theorem for every $N\in\mathbb{N}$ there exists $\xi_N\in(N,N+1)$ and $\zeta_N\in(N,\xi_N)$  such that
\[
\frac{N\big(\varphi(N+1)-\varphi(N)\big)}{\varphi(N)}=\frac{N\varphi'(\xi_N)}{\varphi(N)}=\frac{N\varphi'(N)}{\varphi(N)}+\frac{N(\varphi'(\xi_N)-\varphi'(N))}{\varphi(N)}=\gamma+\theta(N)+\frac{N(\xi_N-N)\varphi''(\zeta_N)}{\varphi(N)}\text{.}
\]
The form of $\varphi''$ in Lemma~2.14 in \cite{MMR} implies that $|x^2\varphi''(x)|\lesssim\varphi(x)$, which together with the fact that $\varphi(2x)\lesssim \varphi(x)$ yields 
\[
\bigg|\frac{N(\xi_N-N)\varphi''(\zeta_N)}{\varphi(N)}\bigg|\lesssim N^{-1}\text{,}
\]
justifying $\ref{phiresults}$, and establishing that the first summand of $\ref{maintermsbp}$ converges to $\gamma L_x$.

By letting $a_n=A_n(f,g)(x)$, the second summand becomes
\[
\frac{1}{|\mathbb{N}_h\cap[N]|}\sum_{n=1}^{N-1}a_nn\Big(\big(\varphi(n+1)-\varphi(n)\big)-\big(\varphi(n+2)-\varphi(n+1)\big)\Big)\text{.}
\] 
We wish to apply Theorem~$\ref{SBParts}$. We let
\[
\widetilde{c}_{N,n}=1_{[N-1]}(n)\frac{n\Big(\big(\varphi(n+1)-\varphi(n)\big)-\big(\varphi(n+2)-\varphi(n+1)\big)\Big)}{\lfloor\varphi(N)\rfloor}\text{,}
\]
and we firstly prove that 
\begin{equation}\label{unweightedlimit}
\lim_{N\to\infty}\sum_{n=1}^N \widetilde{c}_{N,n}=1-\gamma\text{.}
\end{equation}
Let $s(x)=\varphi(x+1)-\varphi(x)$, and by the Mean Valued Theorem there exists $\xi_n\in(n,n+1)$ and $\rho_n\in(\xi_n,\xi_n+1)\subseteq(n,n+2)$ such that
\[
s(n+1)-s(n)=s'(\xi_n)=\varphi'(\xi_n+1)-\varphi'(\xi_n)=\varphi''(\rho_n)
\]
and thus
\[
\sum_{n=1}^N\widetilde{c}_{N,n}=\frac{1}{\lfloor \varphi(N)\rfloor}\sum_{n=1}^{N-1}(-n\varphi''(\rho_n))\text{.}
\]
By the Mean Value Theorem there exists $\tau_n\in(n,n+2)$ such that
\begin{equation}\label{justanestimate}
\bigg|\sum_{n=1}^{N-1}(-n\varphi''(\rho_n))-\sum_{n=1}^{N-1}(-n\varphi''(n))\bigg|\lesssim\sum_{n=1}^N|n\varphi'''(\tau_n)|\lesssim\sum_{n=1}^Nn^{-1}\frac{\varphi(n)}{n}\lesssim \log N  \text{,}
\end{equation}
where for the last estimates we have used the fact that $|x^3\varphi'''(x)|\lesssim\varphi(x)$, as well as the fact that $\lim_{x\to\infty}x^{-1}\varphi(x)=0$ which follow from Lemma~2.6 and Lemma~2.14 in \cite{MMR}. Thus, provided that the limit in the right-hand side below exists, we have
\[
\lim_{N\to\infty}\sum_{n=1}^N\widetilde{c}_{N,n}=\lim_{N\to\infty}\frac{1}{\varphi(N)}\sum_{n=1}^N(-n\varphi''(n))\text{,}\quad\text{since }\lim_{N\to\infty}\frac{\log N}{\varphi(N)}=0\text{.}
\] 
To calculate this limit, we compare with its continuous counterpart, which we can compute as follows
\[
\frac{1}{\varphi(N)}\int_{1}^N(-x\varphi''(x))dx=\frac{-N\varphi'(N)}{\varphi(N)}+\frac{\varphi(N)}{\varphi(N)}+O\big(\varphi(N)^{-1}\big)\text{,}
\]
and using the same argument as in $\ref{phiresults}$, we obtain
\[
\lim_{N\to\infty}\frac{1}{\varphi(N)}\int_{1}^N(-x\varphi''(x))dx=-\gamma+1{.}
\]
We now prove that
\begin{equation}\label{finalgoalreduction}
\lim_{N\to\infty}\frac{1}{\varphi(N)}\bigg|\sum_{n=1}^N(-n\varphi''(n))-\int_{1}^N(-x\varphi''(x))dx\bigg|=0\text{.}
\end{equation}
To see this note that by the Mean Value Theorem
\begin{multline}
\Big|\sum_{n=1}^N(-n\varphi''(n))-\int_{1}^N(-x\varphi''(x))dx\Big|\le\sum_{n=2}^N\int_{n-1}^n|n\varphi''(n)-x\varphi''(x)|dx+O(1)
\\
\le\sum_{n=2}^N\int_{n-1}^n\big(|n(\varphi''(n)-\varphi''(x)|+|(n-x)\varphi''(x)|\big)dx+O(1)
\\
\lesssim\sum_{n=2}^N(n|\varphi'''(\xi_{x,n})|+n^{-2}\varphi(n))+O(1)\lesssim \sum_{n=2}^N n^{-2}\varphi(n)+O(1)\lesssim \log N\text{,}
\end{multline}
where the last estimate is established as in $\ref{justanestimate}$, and $\ref{finalgoalreduction}$ immediately follows. Thus $\ref{unweightedlimit}$ is established.
 
For $c>1$, we get that $\gamma<1$, and we may apply Theorem~$\ref{SBParts}$ for $c_{N,n}=\frac{\widetilde{c}_{N,n}}{1-\gamma}$. The first condition of Theorem~$\ref{SBParts}$ is clearly satisfied and the second condition implies the third since $c_{N,n}\ge 0$ for every $n\gtrsim 1$. In view of $\ref{unweightedlimit}$ it is clear that the second condition also holds and we may apply Theorem~$\ref{SBParts}$ to conclude that $\lim_{N\to\infty}\sum_{n=1}^Nc_{N,n}a_n=L_x$, and thus
 \begin{equation}\label{difred}
\lim_{N\to\infty} \frac{1}{|\mathbb{N}_h\cap[N]|}\sum_{n=1}^{N-1}a_nn\Big(\big(\varphi(n+1)-\varphi(n)\big)-\big(\varphi(n+2)-\varphi(n+1)\big)\Big)=(1-\gamma)L_x\text{.}
 \end{equation}
For $c=1$, we note that by taking into account the fact that $|a_n|=|A_n(f,g)(x)|\le 1$ the second summand may be estimated as follows
\begin{multline}
\frac{1}{|\mathbb{N}_h\cap[N]|}\Big|\sum_{n=1}^{N-1}a_nn\Big(\big(\varphi(n+1)-\varphi(n)\big)-\big(\varphi(n+2)-\varphi(n+1)\big)\Big)\Big|
\\
\lesssim \frac{1}{\lfloor \varphi(N)\rfloor}\sum_{n=1}^{N-1}n\big|\big(\varphi(n+1)-\varphi(n)\big)-\big(\varphi(n+2)-\varphi(n+1)\big)\big|\lesssim \sum_{n=1}^{N}\widetilde{c}_{n,N}+O(\varphi(N)^{-1})\text{,}
\end{multline}
since $\widetilde{c}_{N,n}\ge 0$ for every $n\gtrsim 1$. By $\ref{unweightedlimit}$ together with the estimate above we get 
\[
\lim_{N\to\infty} \frac{1}{|\mathbb{N}_h\cap[N]|}\sum_{n=1}^{N-1}a_nn\Big(\big(\varphi(n+1)-\varphi(n)\big)-\big(\varphi(n+2)-\varphi(n+1)\big)\Big)=0\text{,}
\]
and thus $\ref{difred}$ holds even for $c=1$.

We have calculated the limit of both summands in $\ref{maintermsbp}$, and thus
\[
\lim_{N\to\infty}M_N(f,g)(x)=\gamma L_x+(1-\gamma)L_x= \lim_{N\to\infty}A_N(f,g)(x)\quad\text{for $\mu$-a.e. $x\in X$,}
\] 
as desired. This concludes the proof of $\ref{1Goal}$, which, in turn, concludes the proof of Proposition~$\ref{basicstep1red}$.
\end{proof}
By taking into account Proposition~$\ref{basicstep1red}$, we note that Theorem~$\ref{MTHM}$ will easily follow once we prove that for $c\in[1,23/22)$ and for every $\lambda\in(1,2]$ we have
\begin{equation}\label{ErrorGoal}
\lim_{k\to\infty}E_{\lfloor \lambda^k\rfloor}(f,g)(x)=0\quad\text{for $\mu$-a.e. $x\in X$.}
\end{equation}
We decompose the error term using Lemma~$\ref{TrFourier}$ for a carefully chosen truncation at $M=M(N)$. For the sake of concreteness, we begin by fixing certain parameters:
\begin{equation}\label{choicesforsme}
\varepsilon_0\coloneqq \frac{23-22c}{40c}\text{,}\quad\sigma_0\coloneqq1-\frac{1}{c}+\varepsilon_0\text{,}\quad M\coloneqq \lfloor N^{\sigma_0}\rfloor\text{.}
\end{equation}
Note that $\varepsilon_0>0$, since $c<23/22$. We apply $\ref{splittored}$ for such an $M$ from the aforementioned lemma  to obtain
\begin{multline}
E_N(f,g)(x)=\frac{1}{|\mathbb{N}_h\cap [N]|}\sum_{n\in [N]}\big(\Phi(-\varphi(n+1))-\Phi(-\varphi(n))\big)f(T^nx)g(T^{-n}x)
\\
=\frac{1}{|\mathbb{N}_h\cap [N]|}\sum_{n\in [N]}\bigg(\sum_{0<|m|\le M}\frac{e(m\varphi(n+1))-e(m\varphi(n))}{2\pi i m}\bigg)f(T^nx)g(T^{-n}x)
\\
+\frac{1}{|\mathbb{N}_h\cap [N]|}\sum_{n\in [N]}\big(g_M(-\varphi(n+1))-g_M(-\varphi(n))\big)f(T^nx)g(T^{-n}x)\eqqcolon E^{(1)}_{N}(f,g)(x)+E^{(2)}_{N}(f,g)(x)\text{.}
\end{multline}
The proof of Theorem~$\ref{MTHM}$ will be easily derived once we prove the following.
\begin{proposition}\label{finalreductionE1andE2}Assume $c\in[1,23/22)$ and $h\in\mathcal{R}_c$. Let $(X,\mathcal{B},\mu)$ be a probability space and $T\colon X\to X$ an invertible $\mu$-invariant transformation. Let $f,g$ be 1-bounded functions on $X$. Then the following hold:
\begin{itemize}
\item[\normalfont{(i)}] For every $\lambda\in(1,2]$ we have that $\lim_{k\to\infty}E^{(1)}_{\lfloor \lambda^k\rfloor}(f,g)(x)=0\text{ for $\mu$-a.e. $x\in X$.}$
\item[\normalfont{(ii)}] $\lim_{N\to\infty}E^{(2)}_{N}(f,g)(x)=0\text{ for $\mu$-a.e. $x\in X$.}$
\end{itemize}
\end{proposition}
The second assertion above is substantially easier than the first and we immediately establish it below.
\begin{proof}[Proof of Proposition~$\ref{finalreductionE1andE2}$\normalfont{(ii)}.]
For every pair of $1$-bounded functions $f$, $g$, and for $\mu$-a.e. $x\in X$, we have
\begin{multline}\label{E2quick}
|E_N^{(2)}(f,g)(x)|=\bigg|\frac{1}{|\mathbb{N}_h\cap [N]|}\sum_{n\in [N]}\big(g_M(-\varphi(n+1))-g_M(-\varphi(n))\big)f(T^nx)g(T^{-n}x)\bigg|
\\
\lesssim \frac{1}{\lfloor \varphi(N)\rfloor}\sum_{n\in [N]}\big(|g_M(-\varphi(n+1))|+|g_M(-\varphi(n))|\big)
\\
\lesssim \frac{1}{\varphi(N)}\sum_{n\in[N]}\bigg(\min\bigg\{1,\frac{1}{M\|\varphi(n+1)\|}\bigg\}+\min\bigg\{1,\frac{1}{M\|\varphi(n)\|}\bigg\}\bigg)\text{,}
\end{multline}
where we have used Lemma~$\ref{TrFourier}$. We may use the following lemma to conclude. Let us mention that if $c=1$, then we fix $\sigma$ as in Lemma~2.14 in \cite{MMR}, otherwise we let $\sigma$ be the constant function $1$.
\begin{lemma}\label{Errorinliterature}
Let $c\in[1,4/3)$ and $h\in\mathcal{R}_c$ with $\varphi$ its compositional inverse. Then there exists a positive constant $C=C(h)$ such that for all $N,M\in\mathbb{N}_{\ge 2}$ and $q\in\{0,1\}$ we have
\[
\sum_{n\in[N]}\min\bigg\{1,\frac{1}{M\|\varphi(n+q)\|}\bigg\}\lesssim  \frac{N\log M}{M}+\frac{NM^{1/2}\log N}{(\varphi(N)\sigma(N))^{1/2}}\text{.}
\]
\end{lemma}
\begin{proof}
By $\ref{errorgm}$ we may estimate as follows
\begin{multline}
\sum_{n\in[N]}\min\bigg\{1,\frac{1}{M\|\varphi(n+q)\|}\bigg\}=\sum_{n\in[N]}\sum_{m\in\mathbb{Z}}b_me(m\varphi(n+q))\le\sum_{m\in\mathbb{Z}}|b_m|\Big|\sum_{n\in[N]}e(m\varphi(n+q))\Big|
\\
\lesssim |b_0|N+\sum_{1\le|m|\le M}|b_m|\Big|\sum_{n\in[N]}e(m\varphi(n+q))\Big|+\sum_{|m|\ge M+1}|b_m|\Big|\sum_{n\in[N]}e(m\varphi(n+q))\Big|\text{.}
\end{multline}
By Lemma~3.6 in \cite{WT11}, which is a straightforward application of van der Corput lemma (see Corollary~8.13, page 208 in \cite{IWKO}), we get that for every $1\le P\le P'\le 2P$ we have
\[
\Big|\sum_{P\le n\le P'}e(m\varphi(n+q))\Big|\lesssim |m|^{1/2}P\big(\varphi(P)\sigma(P)\big)^{-1/2}\text{,}
\]
which is eventually increasing in $P$, and thus
\[
\Big|\sum_{1\le n\le N}e(m\varphi(n+q))\Big|\lesssim \log(N)|m|^{1/2}N\big(\varphi(N)\sigma(N)\big)^{-1/2}\text{.}
\]
Therefore, by also taking into account the estimate $|b_m|\lesssim \min\Big\{\frac{\log M}{M},\frac{1}{|m|},\frac{M}{|m|^2}\Big\}$, we get
\begin{multline}
\sum_{n\in[N]}\min\bigg\{1,\frac{1}{M\|\varphi(n+q)\|}\bigg\}
\\
\lesssim \frac{N\log M}{M}+\sum_{1\le|m|\le M}\frac{1}{|m|}\Big(\log(N)|m|^{1/2}N\big(\varphi(N)\sigma(N)\big)^{-1/2}\Big)
\\
+\sum_{|m|\ge M+1}\frac{M}{|m|^2}\Big(\log(N)|m|^{1/2}N\big(\varphi(N)\sigma(N)\big)^{-1/2}\Big)
\\
\lesssim NM^{-1}\log (M) +\log(N)N\big(\varphi(N)\sigma(N)\big)^{-1/2}M^{1/2}
\text{,}
\end{multline}
and the proof is complete.
\end{proof}
Returning to the last line of $\ref{E2quick}$, we see that an application of our lemma yields
\begin{equation}\label{finalboundforE2}
|E_N^{(2)}(f,g)(x)|\lesssim \frac{\log(M)N}{M\varphi(N)}+\frac{\log(N)NM^{1/2}}{\varphi(N)^{3/2}\sigma(N)^{1/2}}\text{.}
\end{equation}
By taking into account the choice of $M$, see $\ref{choicesforsme}$, one can immediately check that the right-hand side converges to $0$ as $N\to\infty$, concluding the proof of Proposition~$\ref{finalreductionE1andE2}$\normalfont{(ii)}. Nevertheless, let us briefly elaborate here.

For the first summand, assuming that $M=\lfloor N^{\sigma_0}\rfloor$, note that it suffices to have that
\[
1-\sigma_0-\frac{1}{c}<0\iff \sigma_0>1-\frac{1}{c}\text{,}
\]
and for the second summand, it suffices to have
\[
1+\frac{\sigma_0}{2}-\frac{3}{2c}<0\iff \sigma_0<\frac{3}{c}-2\text{,}
\]
where for the case $c=1$ we took into account that $\sigma(x)^{-1}\lesssim_{\delta}x^{\delta}$ for all $\delta>0$, see Lemma~2.14 in \cite{MMR}. Any choice $\sigma_0\in\big(1-\frac{1}{c},\frac{3}{c}-2\big)$ is admissible for the present argument to work, and the concrete choice
\[
\sigma_0\coloneqq1-\frac{1}{c}+\frac{23-22c}{40c}\text{,}\quad\text{see $\ref{choicesforsme}$,}
\]
works here since $\sigma_0>1-1/c$ and $\sigma_0<3/c-2$, because 
\[
\frac{23-22c}{40c}>0\quad\text{and}\quad
1-\frac{1}{c}+\frac{23-22c}{40c}<\frac{3}{c}-2\iff c<\frac{137}{98}\text{.}
\] 
The proof is complete.
\end{proof}
It remains to establish Proposition~$\ref{finalreductionE1andE2}$(i), and here a more delicate approach is needed. In the following section we collect some useful intermediate results, and in the final section we provide its proof. Finally, we show that Proposition~$\ref{finalreductionE1andE2}$ together with Proposition~$\ref{basicstep1red}$ immediately yield Theorem~$\ref{MTHM}$.  
\section{Gowers norm Bounds}\label{difSection}
Keeping in mind the choices in $\ref{choicesforsme}$, let us define
\[
 K_N(n)\coloneqq\frac{1_{[N]}(n)}{|\mathbb{N}_h\cap[N]|}\sum_{0<|m|\le M}\frac{e(m\varphi(n+1))-e(m\varphi(n))}{2\pi i m}=\frac{1_{[N]}(n)}{\lfloor \varphi(N)\rfloor}\sum_{0<|m|\le M}\frac{e(m\varphi(n))\psi_m(n)}{2\pi i m}\text{,}
\]
where $\psi_m(n)=e\big(m(\varphi(n+1)-\varphi(n))\big)-1$, so that
\[
E^{(1)}_N(f,g)(x)= \sum_{n\in\mathbb{Z}}K_N(n)f(T^nx)g(T^{-n}x)\text{.}
\]
Proposition~$\ref{finalreductionE1andE2}$(i) will be derived by exploiting appropriate bounds for $\|E^{(1)}_N(f,g)\|_{L_{\mu}^1(X)}$, see Section~$\ref{finSection}$. A key ingredient for establishing these $L^1$-bounds is the following proposition, the proof of which is the content of this section.
\begin{proposition}\label{Sectionskey}Assume $c\in[1,23/22)$ and $h\in\mathcal{R}_c$. Let $S\ge 1$ and $f_0,f_1,f_2\colon\mathbb{Z}\to\mathbb{C}$ be $1$-bounded functions with supports contained in $[-SN,SN]$. Then there exist positive constants $C=C(h,S)$ and $\chi=\chi(h)$ such that for all $N\in\mathbb{N}$ we have
\begin{equation}\label{boundoftheform}
\Big|\sum_{m\in\mathbb{Z}}\sum_{n\in\mathbb{Z}}f_0(m)f_1(m-n)f_2(m+n)K_N(n)\Big|\le CN^{1-\chi} \text{.}
\end{equation}
\end{proposition}
The proof relies on two lemmata. The first one is a simple instance of the fact that difference functions control linear configurations.
\begin{lemma}\label{U3control}Let $S\ge 1$ and $f_0,f_1,f_2,f_3\colon\mathbb{Z}\to\mathbb{C}$ be $1$-bounded functions with supports contained in $[-SN,SN]$. Then there exists a constant $C=C(S)$ such that for all $N\in\mathbb{N}$ we have
\begin{equation}\label{boundU3type}
\Big|\sum_{x\in\mathbb{Z}}\sum_{n\in\mathbb{Z}}f_0(x)f_1(x-n)f_2(x+n)f_3(n)\Big|^8\le C N^{13} 
\sum_{h_3\in\mathbb{Z}}\mu_{N}(h_3)\sum_{h_1,h_2\in[-N,N]}\sum_{n\in\mathbb{Z}}\Delta_{h_1,h_2,h_3}f_3(n)\text{.}
\end{equation}
\end{lemma}
\begin{proof}
The proof is standard and relies on repeated applications of Cauchy-Schwarz, together with van der Corput inequality, see Section~3 in \cite{PelPre}. For the sake of completeness we provide some details. We have
\begin{multline}
\Big|\sum_{x\in\mathbb{Z}}\sum_{n\in\mathbb{Z}}f_0(x)f_1(x-n)f_2(x+n)f_3(n)\Big|^2
\\
\le\Big(\sum_{x\in\mathbb{Z}}|f_0(x)|^2\Big)\Big(\sum_{x\in\mathbb{Z}}\Big|\sum_{n\in\mathbb{Z}}f_1(x-n)f_2(x+n)f_3(n)\Big|^2\Big)
\\
\lesssim N\sum_{x\in\mathbb{Z}}N\sum_{h_1\in\mathbb{Z}}\mu_N(h_1)\sum_{n\in J(h_1)}f_1(x-n)f_2(x+n)f_3(n)\overline{f_1(x-n-h_1)f_2(x+n+h_1)f_3(n+h_1)}
\\
\lesssim N^2\sum_{h_1\in\mathbb{Z}}\mu_N(h_1)\sum_{n\in J(h_1)}\sum_{y\in\mathbb{Z}}f_1(y)f_2(y+2n)f_3(n)\overline{f_1(y-h_1)f_2(y+2n+h_1)f_3(n+h_1)}
\\
\lesssim N^2\sum_{h_1\in\mathbb{Z}}\mu_N(h_1)\sum_{y\in\mathbb{Z}}\Delta_{-h_1}f_1(y)\sum_{n\in J(h_1)}\Delta_{h_1}f_2(y+2n)\Delta_{h_1} f_3(n)\text{,}
\end{multline}
where for the second estimate we have used van der Corput inequality, see Lemma~3.1 in \cite{Pre} or Lemma~3.1 in \cite{PelPre}, and where $J(h_1)\coloneqq [N]\cap([N]-h_1)$.
We repeat the procedure
\begin{multline}
\Big|\sum_{x\in\mathbb{Z}}\sum_{n\in\mathbb{Z}}f_0(x)f_1(x-n)f_2(x+n)f_3(n)\Big|^4
\\
\lesssim N^4\Big|\sum_{h_1\in\mathbb{Z}}\mu_N(h_1)\sum_{y\in\mathbb{Z}}\Delta_{-h_1}f_1(y)\sum_{n\in J(h_1)}\Delta_{h_1}f_2(y+2n)\Delta_{h_1} f_3(n)\Big|^2
\\
\lesssim N^4\Big(\sum_{h_1\in\mathbb{Z}}\mu_N(h_1)^2\Big)\Big(\sum_{h_1\in[-N,N]}\Big|\sum_{y\in\mathbb{Z}}\Delta_{-h_1}f_1(y)\sum_{n\in J(h_1)}\Delta_{h_1}f_2(y+2n)\Delta_{h_1} f_3(n)\Big|^2\Big)
\\
\lesssim  N^3\sum_{h_1\in[-N,N]}\Big|\sum_{y\in\mathbb{Z}}\Delta_{-h_1}f_1(y)\sum_{n\in J(h_1)}\Delta_{h_1}f_2(y+2n)\Delta_{h_1} f_3(n)\Big|^2
\\
\lesssim  N^3\sum_{h_1\in[-N,N]}\Big(\sum_{y\in\mathbb{Z}}|\Delta_{-h_1}f_1(y)|^2\Big)\Big(\sum_{y\in\mathbb{Z}}\Big|\sum_{n\in J(h_1)}\Delta_{h_1}f_2(y+2n)\Delta_{h_1} f_3(n)\Big|^2\Big)
\\
\lesssim  N^4\sum_{h_1\in[-N,N]}\sum_{y\in\mathbb{Z}}\Big|\sum_{n\in J(h_1)}\Delta_{h_1}f_2(y+2n)\Delta_{h_1} f_3(n)\Big|^2
\\
\lesssim N^5\sum_{h_1\in[-N,N]}\sum_{y\in\mathbb{Z}}\sum_{h_2\in\mathbb{Z}}\mu_N(h_2)\sum_{n\in J(h_1,h_2)}\Delta_{h_1,h_2}f_3(n)\cdot
\\
\cdot f_2(y+2n)\overline{f_2(y+2n+h_1)f_2(y+2n+2h_2)}f_2(y+2n+2h_2+h_1)
\\
\lesssim N^5\sum_{h_1\in[-N,N]}\sum_{h_2\in\mathbb{Z}}\mu_N(h_2)\sum_{n\in J(h_1,h_2)}\sum_{x\in\mathbb{Z}}\Delta_{h_1,h_2}f_3(n)\Delta_{h_1,2h_2}f_2(x)\text{,}
\end{multline}
where $J(h_1,h_2)=J(h_1)\cap (J(h_1)-h_2)$.
We repeat one final time
\begin{multline}
\Big|\sum_{x\in\mathbb{Z}}\sum_{n\in\mathbb{Z}}f_0(x)f_1(x-n)f_2(x+n)f_3(n)\Big|^8
\\
\lesssim N^{10}\Big|\sum_{h_1\in[-N,N]}\sum_{h_2\in\mathbb{Z}}\mu_N(h_2)\sum_{x\in\mathbb{Z}}\Delta_{h_1,2h_2}f_2(x)\sum_{n\in J(h_1,h_2)}\Delta_{h_1,h_2}f_3(n)\Big|^2
\\
\lesssim N^{10}\Big(\sum_{h_1\in[-N,N]}\sum_{h_2\in\mathbb{Z}}\mu_N(h_2)^2\Big)\Big(\sum_{h_1,h_2\in[-N,N]}\Big|\sum_{x\in\mathbb{Z}}\Delta_{h_1,2h_2}f_2(x)\sum_{n\in J(h_1,h_2)}\Delta_{h_1,h_2}f_3(n)\Big|^2\Big)
\\
\lesssim N^{12}\sum_{h_1,h_2\in[-N,N]}\Big|\sum_{n\in J(h_1,h_2)}\Delta_{h_1,h_2}f_3(n)\Big|^2
\\
\lesssim N^{13}\sum_{h_1,h_2\in[-N,N]}\sum_{h_3\in\mathbb{Z}}\mu_N(h_3)\sum_{n\in J(h_1,h_2,h_3)}\Delta_{h_1,h_2,h_3}f_3(n)\text{,}
\end{multline}
where $J(h_1,h_2,h_3)=J(h_1,h_2)\cap(J(h_1,h_2)-h_3)$. Taking into account the support of $\Delta_{h_1,h_2,h_3}f_3(n)$, the restriction in the final summation can be discarded. Thus, we get
\[
\Big|\sum_{x\in\mathbb{Z}}\sum_{n\in\mathbb{Z}}f_0(x)f_1(x-n)f_2(x+n)f_3(n)\Big|^8\le C N^{13} 
\sum_{h_3\in\mathbb{Z}}\mu_{N}(h_3)\sum_{h_1,h_2\in[-N,N]}\sum_{n\in\mathbb{Z}}\Delta_{h_1,h_2,h_3}f_3(n)\text{, as desired.}
\] 
\end{proof}
The second lemma will ultimately allow us to establish Proposition~$\ref{Sectionskey}$ by essentially providing a bound for the $U^3$-norm of a dyadic variant of our kernel. For every $l\in\mathbb{N}_0$ and $N\in\mathbb{N}$, let 
\[
K_{N,l}(n)=1_{[2^l,\min(2^{l+1},N+1))}(n)K_{N}(n)\text{.}
\]We remind the reader that $M$ is defined in $\ref{choicesforsme}$.
\begin{lemma}\label{smallgain}
Let $c\in[1,23/22)$, $h\in\mathcal{R}_c$ and $\kappa\in(0,1]$. Then there exists a positive constant $C=C(h,\kappa)$ such that for every $N\in\mathbb{N}$ and $l\in[0,\log_2(N+1)]\cap\mathbb{Z}$ we have 
\[
\sum_{h_3\in\mathbb{Z}}\mu_{N}(h_3)\sum_{h_1,h_2\in[-N,N]}\sum_{n\in\mathbb{Z}}\Delta_{h_1,h_2,h_3}K_{N,l}(n)\le C \Big(N^2\varphi(N)^{-8+\kappa}+N\varphi(N)^{-8}2^{-\frac{4l}{3}}\sigma(2^l)^{-\frac{2}{3}}\varphi(2^l)^{\frac{10-2\kappa}{3}}M^{\frac{16}{3}}\Big)\text{.}
\]
\end{lemma}
\begin{proof}
It will be more convenient to work with an unnormalized kernel, so let us define
\begin{equation}\label{LNdef}
L_{N,l}(n)=\lfloor \varphi(N)\rfloor K_{N,l}(n)=1_{[2^l,\min(2^{l+1},N+1))}(n)\sum_{0<|m|\le M}\frac{e(m\varphi(n+1))-e(m\varphi(n))}{2\pi i m}\text{.}
\end{equation}
By Lemma~$\ref{TrFourier}$ we get that $|L_{N,l}(n)|\lesssim 1$, and the implied constant is absolute. For any $h_3\in\mathbb{Z}$ we have
\[
\sum_{|h_1|,|h_2|\le N}\sum_{n\in\mathbb{Z}}\Delta_{h_1,h_2,h_3}K_{N,l}(n)=\sum_{x,h_1,h_2\in\mathbb{Z}}\Delta_{h_1,h_2}\big(\Delta_{h_3}K_{N,l}\big)(x)=\|\Delta_{h_3}K_{N,l}\|_{U^2}^4=\lfloor \varphi(N)\rfloor^{-8}\|\Delta_{h_3}L_{N,l}\|_{U^2}^4\text{.}
\]
By the inverse theorem for the $U^2$-norm, see Lemma~A.1 in \cite{PelPre}, there exists $\xi_{h_{3},N,l}\in\mathbb{T}$ such that
\[
\|\Delta_{h_3}L_{N,l}\|_{U^2}^4\lesssim N \Big|\sum_{x\in\mathbb{Z}}\big(\Delta_{h_3}L_{N,l}(x)\big)e(x\xi_{h_{3},N,l})\Big|^2\text{,}\quad \text{where we have used that $|L_{N,l}(n)|\lesssim 1$.}
\]
We have shown that for each $h_3\in[-N,N]$ and $l\in[0,\log_2(N+1)]\cap\mathbb{Z}$ there exists $\xi_{h_{3},N,l}\in\mathbb{T}$ such that
\begin{equation}\label{Return}
\sum_{h_3\in\mathbb{Z}}\mu_{N}(h_3)\sum_{h_1,h_2\in[-N,N]}\sum_{n\in\mathbb{Z}}\Delta_{h_1,h_2,h_3}K_{N,l}(n)\lesssim \varphi(N)^{-8}N \sum_{h_3\in\mathbb{Z}}\mu_{N}(h_3)\Big|\sum_{x\in\mathbb{Z}}\big(\Delta_{h_3}L_{N,l}(x)\big)e(x\xi_{h_{3},N,l})\Big|^2\text{,}
\end{equation}
and it will suffice to appropriately bound $\sup_{\xi\in\mathbb{T}}\big|\sum_{x\in\mathbb{Z}}\big(\Delta_{h_3}L_{N,l}(x)\big)e(x\xi)\big|$ for most $h_3$'s. More precisely, we establish the following to conclude.
\begin{lemma}\label{wt11est}Let $c\in[1,23/22)$, $h\in\mathcal{R}_c$ and $\kappa\in(0,1]$. Then there exists a positive constant $C=C(h,\kappa)$ such that for every $N\in\mathbb{N}$, $l\in[0,\log_2(N+1)]\cap\mathbb{Z}$ and $h_3\in\mathbb{Z}$ with $|h_3|\ge \varphi(2^l)^\kappa$, we get 
\[
\Big\|\sum_{x\in\mathbb{Z}}\big(\Delta_{h_3}L_{N,l}(x)\big)e(x\xi)\Big\|_{L_{d\xi}^{\infty}(\mathbb{T})}\le C2^{-\frac{2l}{3}}\sigma(2^l)^{-\frac{1}{3}}\varphi(2^l)^{\frac{5-\kappa}{3}}M^{\frac{8}{3}}\text{.}
\] 
\end{lemma}
\begin{proof}
Let $\xi\in\mathbb{T}$, $N\in\mathbb{N}$, $\kappa\in(0,1]$, $l\in[0,\log_2(N+1)]\cap\mathbb{Z}$ and assume that $h_3\in\mathbb{Z}$ with $|h_3|\ge \varphi(2^l)^\kappa$. Note that
\begin{multline}
\big(\Delta_{h_3}L_{N,l}(n)\big)e(n\xi)=L_{N,l}(n)\overline{L_{N,l}(n+h_3)}e(n\xi)
\\
=1_{2^l\le n, n+h_3<\min(2^{l+1},N+1)}\sum_{0<|m_1|\le M}\frac{e(m_1 \varphi(n))\psi_{m_1}(n)}{2\pi i m_1}\sum_{0<|m_2|\le M}\frac{e(-m_2 \varphi(n+h_3))\overline{\psi_{m_2}(n+h_3)}}{-2\pi i m_2}e(n\xi)
\\
=1_{2^l\le n, n+h_3<\min(2^{l+1},N+1)}\sum_{0<|m_1|,|m_2|\le M}\frac{e(m_1 \varphi(n)-m_2 \varphi(n+h_3)+n\xi)\psi_{m_1}(n)\overline{\psi_{m_2}(n+h_3)}}{4\pi^2 m_1m_2}\text{,}
\end{multline}
where $\psi_m(n)=e\big(m(\varphi(n+1)-\varphi(n))\big)-1$.
Thus we get
\begin{multline}\label{h_3positive}
\Big|\sum_{n\in\mathbb{Z}}\big(\Delta_{h_3}L_{N,l}(n)\big)e(n\xi)\Big|
\\
\lesssim
\sum_{0<|m_1|,|m_2|\le M}\frac{1}{|m_1m_2|}\Big|\sum_{n,n+h_3\in[2^l,\min(2^{l+1},N+1))}e(m_1 \varphi(n)-m_2 \varphi(n+h_3)+n\xi)\psi_{m_1}(n)\overline{\psi_{m_2}(n+h_3)}\Big|\text{.}
\end{multline}
We firstly assume that $h_3>0$. We apply Corollary~3.12 from \cite{WT11}, which yields the following estimate
\begin{multline}
\bigg|\sum_{n,n+h_3\in[2^l,\min(2^{l+1},N+1))}e(m_1 \varphi(n)-m_2 \varphi(n+h_3)+n\xi)\psi_{m_1}(n)\overline{\psi_{m_2}(n+h_3)}\bigg|
\\
\lesssim \max\{|m_1|,|m_2|\}^{2/3}2^{4l/3}\sigma(2^l)^{-1/3}\varphi(2^l)^{-(1+\kappa)/3}\Big(|m_1m_2|\varphi(2^l)^22^{-2l}+2^l|m_1m_2|\varphi(2^l)^22^{-3l}\Big)
\\
=|m_1m_2|\max\{|m_1|,|m_2|\}^{2/3}2^{-2l/3}\sigma(2^l)^{-1/3}\varphi(2^l)^{(5-\kappa)/3}
\text{,}
\end{multline}
where the bound can be derived through standard estimates for $\psi_m$, see page 37 in \cite{WT11LD} for detailed calculations. Returning to $\ref{h_3positive}$ we get
\begin{equation}\label{h3pored}
\Big|\sum_{n\in\mathbb{Z}}\big(\Delta_{h_3}L_{N,l}(n)\big)e(n\xi)\Big|\lesssim
\sum_{0<|m_1|,|m_2|\le M}\max\{|m_1|,|m_2|\}^{2/3}2^{-2l/3}\sigma(2^l)^{-1/3}\varphi(2^l)^{(5-\kappa)/3}\text{.}
\end{equation}
We note that for $h_3<0$ we may perform a change of variables in $\ref{h_3positive}$ to obtain
\begin{multline}\label{h_3negative}
\Big|\sum_{n\in\mathbb{Z}}\big(\Delta_{h_3}L_{N,l}(n)\big)e(n\xi)\Big|
\\
\lesssim
\sum_{0<|m_1|,|m_2|\le M}\frac{1}{|m_1||m_2|}\bigg|\sum_{k,k-h_3\in[2^l,\min(2^{l+1},N+1))}e(m_1 \varphi(k-h_3)-m_2 \varphi(k)+n\xi)\psi_{m_1}(k-h_3)\overline{\psi_{m_2}(k)}\bigg|\text{,}
\end{multline}
and one may apply the same argument as before with $h_3$ replaced by $-h_3>0$, resulting in the same bound $\ref{h3pored}$. Now we may estimate as follows 
\begin{multline}\label{h3porednew}
\Big|\sum_{n\in\mathbb{Z}}\big(\Delta_{h_3}L_{N,l}(n)\big)e(n\xi)\Big|\lesssim
2^{-2l/3}\sigma(2^l)^{-1/3}\varphi(2^l)^{(5-\kappa)/3}\sum_{0<|m_1|,|m_2|\le M}\max\{|m_1|,|m_2|\}^{2/3}
\\
\lesssim
2^{-\frac{2l}{3}}\sigma(2^l)^{-\frac{1}{3}}\varphi(2^l)^{\frac{5-\kappa}{3}}M^{\frac{8}{3}}\text{,}
\end{multline}
and the proof is complete.
\end{proof}
Returning  back to $\ref{Return}$, one concludes as follows
\begin{multline}
\sum_{h_3\in\mathbb{Z}}\mu_{N}(h_3)\sum_{h_1,h_2\in[-N,N]}\sum_{n\in\mathbb{Z}}\Delta_{h_1,h_2,h_3}K_{N,l}(n)\lesssim \frac{N}{\varphi(N)^{8}} \sum_{h_3\in[-N,N]}\mu_{N}(h_3)\Big|\sum_{x\in\mathbb{Z}}\big(\Delta_{h_3}L_{N,l}(x)\big)e(x\xi_{h_{3},N,l})\Big|^2
\\
\le N\varphi(N)^{-8} \sum_{h_3\in[-\varphi(2^l)^{\kappa},\varphi(2^l)^{\kappa}]}\mu_{N}(h_3)\Big|\sum_{x\in\mathbb{Z}}\big(\Delta_{h_3}L_{N,l}(x)\big)e(x\xi_{h_{3},N,l})\Big|^2
\\
+N\varphi(N)^{-8} \sum_{|h_3|\in(\varphi(2^l)^{\kappa},N]}\mu_{N}(h_3)\Big|\sum_{x\in\mathbb{Z}}\big(\Delta_{h_3}L_{N,l}(x)\big)e(x\xi_{h_{3},N,l})\Big|^2
\\
\lesssim N\varphi(N)^{-8}\varphi(N)^{\kappa}N^{-1}N^2+N\varphi(N)^{-8}2^{-\frac{4l}{3}}\sigma(2^l)^{-\frac{2}{3}}\varphi(2^l)^{\frac{10-2\kappa}{3}}M^{\frac{16}{3}}
\\
N^2\varphi(N)^{-8+\kappa}+N\varphi(N)^{-8}2^{-\frac{4l}{3}}\sigma(2^l)^{-\frac{2}{3}}\varphi(2^l)^{\frac{10-2\kappa}{3}}M^{\frac{16}{3}}
\text{.}
\end{multline}
This concludes the proof of Lemma~$\ref{smallgain}$.
\end{proof}
We now are ready to prove Proposition~$\ref{Sectionskey}$.
\begin{proof}[Proof of Proposition $\ref{Sectionskey}$]
By Lemma~$\ref{U3control}$ and Lemma~$\ref{smallgain}$ we get
\begin{multline}\label{dyadicfinalest}
\Big|\sum_{n\in\mathbb{Z}}\sum_{x\in\mathbb{Z}}f_0(x)f_1(x-n)f_2(x+n)K_N(n)\Big|
\\
\lesssim\sum_{0\le l\le\log_2(N+1)}\Big|\sum_{n\in\mathbb{Z}}\sum_{x\in\mathbb{Z}}f_0(x)f_1(x-n)f_2(x+n)\big(K_N(n)1_{[2^l,\min(2^{l+1},N+1))}(n)\big)\Big|
\\
\lesssim\sum_{0\le l\le\log_2(N+1)}N^{\frac{13}{8}}\Big( 
\sum_{h_3\in\mathbb{Z}}\mu_{N}(h_3)\sum_{h_1,h_2\in[-N,N]}\sum_{n\in\mathbb{Z}}\Delta_{h_1,h_2,h_3}K_{N,l}(n)\Big)^{\frac{1}{8}}
\\
\lesssim \sum_{0\le l\le\log_2(N+1)}N^{\frac{13}{8}} \Big(N^{\frac{2}{8}}\varphi(N)^{-1+\frac{\kappa}{8}}+N^{\frac{1}{8}}\varphi(N)^{-1}2^{-\frac{l}{6}}\sigma(2^l)^{-\frac{1}{12}}\varphi(2^l)^{\frac{5-\kappa}{12}}M^{\frac{2}{3}}\Big)
\\
\lesssim \log(N)N^{\frac{15}{8}}\varphi(N)^{-1+\frac{\kappa}{8}}+\sum_{0\le l\le\log_2(N+1)}N^{\frac{14}{8}} \varphi(N)^{-1}2^{-\frac{l}{6}}\sigma(2^l)^{-1/12}\varphi(2^l)^{\frac{5-\kappa}{12}}M^{\frac{2}{3}}
\\
\lesssim N\big(\log(N)N^{\frac{7}{8}}\varphi(N)^{-1+\frac{\kappa}{8}}\big)+N^{\frac{7}{4}}\varphi(N)^{-1}M^{\frac{2}{3}} \sum_{0\le l\le\log_2(N+1)}\sigma(2^l)^{-\frac{1}{12}}2^{-\frac{l}{6}}\varphi(2^l)^{\frac{5-\kappa}{12}}\text{.}
\end{multline}
We choose $\kappa\coloneqq \frac{9c-6}{5}$, note that this is possible since $\frac{9c-6}{5}<1\iff c<11/9$. For the first summand above we note that
\[
\frac{7}{8}-\frac{1}{c}+\frac{\kappa}{8c}=\frac{22c-23}{20c}<0\text{,}
\]
since $c<23/22$ and thus $N\big(\log(N)N^{\frac{7}{8}}\varphi(N)^{-1+\frac{\kappa}{8}}\big)\lesssim N^{1-\chi}$ for some $\chi=\chi(c)>0$. For the second summand firstly note that for all $\varepsilon>0$ we have 
\begin{multline}
\sum_{0\le l\le\log_2(N+1)}\sigma(2^l)^{-\frac{1}{12}}2^{-\frac{l}{6}}\varphi(2^l)^{\frac{5-\kappa}{12}}\lesssim_{\varepsilon}\sum_{0\le l\le\log_2(N+1)} 2^{l\varepsilon}2^{-\frac{l}{6}}((2^l)^{\frac{1}{c}+\varepsilon})^{\frac{5-\kappa}{12}}
\\
=\sum_{0\le l\le\log_2(N+1)} 2^{l\big(\varepsilon-\frac{1}{6}+(\frac{1}{c}+\varepsilon)(\frac{5-\kappa}{12})\big)}\lesssim N^{\big(\varepsilon-\frac{1}{6}+(\frac{1}{c}+\varepsilon)(\frac{5-\kappa}{12})\big)}\lesssim N^{-\frac{1}{6}+\frac{5}{12c}-\frac{\kappa}{12c}+2\varepsilon}\text{.}
\end{multline}
Note that the penultimate estimate holds since
\begin{equation}\label{numerologyforsecondest}
-\frac{1}{6}+\frac{5}{12c}-\frac{\kappa}{12c}=\frac{-19c+31}{60c}>0\text{,}
\end{equation}
since $c<31/19$. Thus applying the above estimate for $\varepsilon=\varepsilon_0/12$ yields the following bounds for the second summand in the last line of $\ref{dyadicfinalest}$
\begin{multline}
N^{\frac{7}{4}}\varphi(N)^{-1}M^{\frac{2}{3}} \sum_{0\le l\le\log_2(N+1)}\sigma(2^l)^{-\frac{1}{12}}2^{-\frac{l}{6}}\varphi(2^l)^{\frac{5-\kappa}{12}}\lesssim N^{\frac{7}{4}}N^{-\frac{1}{c}+\frac{\varepsilon_0}{6}}N^{\frac{2}{3}-\frac{2}{3c}+\frac{2\varepsilon_0}{3}}N^{-\frac{1}{6}+\frac{5}{12c}-\frac{\kappa}{12c}+\frac{\varepsilon_0}{6}} 
\\
=NN^{\frac{22c-23}{20c}+\varepsilon_0}=NN^{\frac{44c-46}{40c}+\frac{23-22c}{40c}}=NN^{\frac{22c-23}{40c}}\text{.}
\end{multline}
Since $\frac{22c-23}{40c}<0$, combining the estimates for both summands in the last line of $\ref{dyadicfinalest}$ we get that there exists $\chi=\chi(c)>0$ such that 
\[
\Big|\sum_{n\in\mathbb{Z}}\sum_{x\in\mathbb{Z}}f_0(x)f_1(x-n)f_2(x+n)K_N(n)\Big|\lesssim N^{1-\chi}\text{,}
\]
as desired, and the proof is complete.
\end{proof}
\section{Concluding the proof of Theorem~$\ref{MTHM}$}\label{finSection}
Here we explain how to use Proposition~$\ref{Sectionskey}$ to establish the first assertion of Proposition~$\ref{finalreductionE1andE2}$ and conclude the proof of Theorem~$\ref{MTHM}$. Bounds of the form $\ref{boundoftheform}$ immediately yield the $L^1$-bound in $\ref{L1saving}$, which, in turn, will yield the desired result. The following proposition should be understood as a simple instance of Calder\'on's transference principle utilizing Proposition~$\ref{Sectionskey}$.
\begin{proposition}\label{L1controlEN}Assume $c\in[1,23/22)$ and $h\in\mathcal{R}_c$. Let $(X,\mathcal{B},\mu)$ be a probability space and $T\colon X\to X$ an invertible $\mu$-invariant transformation. Let $f,g\colon X\to\mathbb{C}$ be 1-bounded functions. Then there exist positive constants $C=C(h)$ and $\chi=\chi(h)$ such that
\begin{equation}\label{L1saving}
\|E_N^{(1)}(f,g)\|_{L^1_{\mu}(X)}\le CN^{-\chi}\text{.}
\end{equation}
\end{proposition}
\begin{proof}
By duality and $\mu$-invariance there exists a 1-bounded function $l$ such that
\begin{multline}
\|E^{(1)}_N(f,g)\|_{L^1_{\mu}(X)}=\int_Xl(x)E^{(1)}_N(f,g)(x)d\mu(x)=\int_Xl(x)\sum_{n\in\mathbb{Z}}K_N(n)f(T^nx)g(T^{-n}x)d\mu(x)
\\
=\frac{1}{N}\sum_{m\in[N]}\int_Xl(T^mx)\sum_{n\in\mathbb{Z}}K_N(n)f(T^{n+m}x)g(T^{-n+m}x)d\mu(x)
\\
\le\int_X\frac{1}{N}\Big|\sum_{m\in\mathbb{Z}}\sum_{n\in\mathbb{Z}}\big(l(T^mx)1_{[N]}(m)\big)\big(f(T^{n+m}x)1_{[-2N,2N]}(n+m)\big)\cdot
\\
\cdot\big(g(T^{m-n}x)1_{[-2N,2N]}(m-n)\big)K_N(n)\Big|d\mu(x)\text{.}
\end{multline}
For every fixed $x\in X$, we apply Proposition~$\ref{Sectionskey}$ to the obvious functions to conclude that
\begin{multline}
\int_X\frac{1}{N}\Big|\sum_{m\in\mathbb{Z}}\sum_{n\in\mathbb{Z}}\big(l(T^mx)1_{[N]}(m)\big)\big(f(T^{n+m}x)1_{[-2N,2N]}(n+m)\big)\cdot
\\
\cdot\big(g(T^{m-n}x)1_{[-2N,2N]}(m-n)\big)K_N(n)\Big|d\mu(x)\lesssim N^{-\chi}\text{,} 
\end{multline}
and the proof is complete.
\end{proof}
Finally, we explain how to use Proposition~$\ref{L1controlEN}$ to prove Proposition~$\ref{finalreductionE1andE2}$(i), as well as how to conclude the proof of Theorem~$\ref{MTHM}$.
\begin{proof}[Proof of Proposition~$\ref{finalreductionE1andE2}$\normalfont{(i)}.] Fix $\lambda\in(1,2]$; using Proposition~$\ref{L1controlEN}$ it is easy to see that
\[
\Big\|\sum_{k\in\mathbb{N}_0}\big|E^{(1)}_{\lfloor \lambda^k\rfloor}(f,g)\big|\Big\|_{L^1_{\mu}(X)}\le\sum_{k\in\mathbb{N}_0}\|E^{(1)}_{\lfloor \lambda^k\rfloor}(f,g)\|_{L^1_{\mu}(X)}\lesssim \sum_{k\in\mathbb{N}_0}\lambda^{-\chi k}<\infty\text{,}
\]
and thus $\sum_{k\in\mathbb{N}_0}\big|E_{\lfloor \lambda^k\rfloor}(f,g)(x)\big|<\infty$ for $\mu$-a.e. $x\in X$, which, in turn, implies that
\[
\lim_{k\to\infty}E_{\lfloor \lambda^k\rfloor}(f,g)(x)=0\quad\text{for $\mu$-a.e. $x\in X$, as desired.}
\]
\end{proof}
\begin{proof}[Proof of Theorem~$\ref{MTHM}$.] It suffices to prove that for every $\lambda\in(1,2]$ we have
\begin{equation}\label{lambdavariant}
\lim_{k\to\infty}B_{\lfloor\lambda^k\rfloor}(f,g)(x)=\lim_{k\to\infty}A_{\lfloor\lambda^k\rfloor}(f,g)(x)\quad\text{for $\mu$-a.e. $x\in X$.}
\end{equation}
We note that the argument in the proof of Proposition~$\ref{basicstep1red}$ together with Proposition~$\ref{finalreductionE1andE2}$ yield that for $\mu$-a.e. $x\in X$ we have
\begin{multline}
\limsup_{k\to\infty}\big|B_{\lfloor\lambda^k\rfloor}(f,g)(x)-A_{\lfloor\lambda^k\rfloor}(f,g)(x)\big|\le\limsup_{k\to\infty}|E_{\lfloor\lambda^k\rfloor}(f,g)(x)|
\\
\le \limsup_{k\to\infty}|E^{(1)}_{\lfloor\lambda^k\rfloor}(f,g)(x)|+\limsup_{k\to\infty}|E^{(2)}_{\lfloor\lambda^k\rfloor}(f,g)(x)|=0\text{,}
\end{multline}
and thus $\ref{lambdavariant}$ is established and the proof of Theorem~$\ref{MTHM}$ is complete.
\end{proof}

\end{document}